\documentclass[11pt]{amsart}   % includes 11pt font size (default 10pt)

\usepackage{verbatim}
\usepackage{amssymb,amsthm,amsmath,amsfonts}
\usepackage{color}
\definecolor{darkblue}{rgb}{0, 0, .4}
\definecolor{grey}{rgb}{.7, .7, .7}

\usepackage{ifpdf}
\ifpdf
  \usepackage[pdftex]{epsfig}

  \usepackage{lscape}

  \usepackage[breaklinks]{hyperref}
  \hypersetup{
            colorlinks=true,
            linkcolor=darkblue,
            anchorcolor=darkblue,
            citecolor=darkblue,
            pagecolor=darkblue,
            urlcolor=darkblue,
            pdftitle={},
            pdfauthor={}
  }
\else
  \usepackage[dvips]{epsfig}
  \usepackage{lscape}

  \newcommand{\href}[2]{#2}
  \newcommand{\url}[2]{#2}
\fi

\newtheorem{theorem}{Theorem}[section]
\newtheorem{lemma}[theorem]{Lemma}

\theoremstyle{definition}
\newtheorem{definition}[theorem]{Definition}
\newtheorem{example}[theorem]{Example}

\theoremstyle{remark}
\newtheorem{remark}[theorem]{Remark}

\numberwithin{equation}{section}

\theoremstyle{theorem}
\newtheorem{corollary}[theorem]{Corollary}

\newcommand{\N}[0]{\mathbb{N}}

\usepackage{graphicx}
\input xy 
\xyoption{all}

\usepackage{tikz}
\usetikzlibrary{trees}
\usepackage{ifthen}

\usetikzlibrary{patterns}

\newcommand{\ra}{\rightarrow}

\newcommand{\T}{\mathcal{T}}
\renewcommand{\S}{\mathcal{S}}
\newcommand{\SN}{\mathfrak{S}}

\renewcommand{\c}{\circ}
\newcommand{\pf}[1]{\pi|_{[#1]}}

\newcommand{\G}[0]{\Gamma}

\newcommand{\stir}[1]{\{#1\}!}
\newcommand{\ta}[1]{[#1]}
\newcommand{\fa}[1]{[#1]!}
\renewcommand{\t}[0]{\theta}
\renewcommand{\k}[0]{\kappa}
\newcommand{\DW}[0]{\Delta W}

\begin{document}

\title{Weighted games of best choice}

\begin{abstract}
The game of best choice (also known as the secretary problem) is a model for sequential decision making with a long history and many variations.  The classical setup assumes that the sequence of candidate rankings are uniformly distributed.  Given a statistic on the symmetric group, one can instead weight each permutation according to an exponential function in the statistic.  We play the game of best choice on the Ewens and Mallows distributions that are obtained in this way from the number of left-to-right maxima and number of inversions in the permutation, respectively.  For each of these, we give the optimal strategy and probability of winning.  Moreover, we introduce a general class of permutation statistics that always produces games of best choice whose optimal strategies are positional, which simplifies their analysis considerably.
\end{abstract}

\author{Brant Jones}
\address{Department of Mathematics and Statistics, MSC 1911, James Madison University, Harrisonburg, VA 22807}
\email{\href{mailto:jones3bc@jmu.edu}{\texttt{jones3bc@jmu.edu}}}
\urladdr{\url{http://educ.jmu.edu/\~jones3bc/}}

%\keywords{}

\date{\today}

\maketitle

%%%%%%%%%%%%%%%%%%%%%%%%%%%%%%%%%%%%%%%%%%%%%%%%%%%%%%%%%%%%%%%%%%%%%
%  Section
%%%%%%%%%%%%%%%%%%%%%%%%%%%%%%%%%%%%%%%%%%%%%%%%%%%%%%%%%%%%%%%%%%%%%
\section{Introduction}\label{s:intro}

The game of best choice (or secretary problem) is a model for sequential
decision making, popularized in a Scientific American column of Martin Gardner
reprinted in \cite{gardner}.  In the simplest variant, an {\bf interviewer} evaluates $N$
{\bf candidates} one by one.  After each interview, the interviewer ranks the
current candidate against all of the candidates interviewed so far, and decides
whether to {\bf accept} the current candidate (ending the game) or to
{\bf reject} the current candidate (in which case, they cannot be recalled
later).  The goal of the game is to hire the best candidate out of $N$.  It
turns out that the optimal strategy is to reject an initial set of candidates,
of size $N/e$ when $N$ is large, and use them as a training set by hiring the
next candidate who is better than all of them (or the last candidate if no
subsequent candidate is better).  The probability of hiring the best candidate
out of $N$ with this strategy approaches $1/e$.  This result has been
generalized in several directions by \cite{gilbert--mosteller}; see also
\cite{ferguson} and \cite{freeman} for historical surveys.  Recently,
researchers (e.g. \cite{kleinberg08}) have begun applying the best-choice
framework to online auctions where the ``candidate rankings'' are bids (that may
arrive and expire at different times) and the player must choose which bid to
accept, ending the auction.  

Most models assume that all $N!$ interview rank orders are equally
likely, which we believe is mathematically expedient but unrealistic.  There
may exist extrinsic trends in the candidate pool due to changes in general
economic conditions over the period that the player is conducting the $N$
interviews.  Also, as the interviewer compares and ranks the candidates at each
step, they are forced to find distinctions that may simultaneously be used to
hone the pool by filtering out irrelevant candidates.  Overall, this intrinsic
learning about the candidate pool would tend to result in candidate ranks that are
improving over time rather than uniform.

Towards understanding such mechanisms, we are interested in how assumptions
about the distribution of interview rank orders change the optimal strategy and
probability of success in the game of best choice.  While many authors have
investigated a ``full-information'' version of the game in which the interviewer
observes values from a known distribution, only a few papers from the literature
(e.g.  \cite{pfeifer,reeves--flack}) have previously considered such nonuniform
rank distributions for the secretary problem.  Continuing work from
\cite{fowlkes} and \cite{jones18}, we establish in this paper a class of
weighted models that generalize the classical game in a natural way.

We model interview orderings as {\bf permutations}.  The permutation $\pi$ of
$N$ is expressed in one-line notation as $[\pi_1 \pi_2 \cdots \pi_N]$
where the $\pi_i$ consist of the elements $1, 2, \ldots, N$ (so each element
appears exactly once).  In the best choice game, $\pi_i$ is the rank of the
$i$th candidate interviewed {\em in reality}, where rank $N$ is best and $1$ is
worst.  What the interviewer sees at each step, however, are {\em relative} rankings.
For example, corresponding to the interview order $\pi = [2516374]$, the interviewer
sees the sequence of permutations $1, 12, 231, 2314, 24153, 241536, 2516374$
and must use only this information to determine when to accept a candidate,
ending the game.  The {\bf left-to-right maxima} of $\pi$ consist of the
elements $\pi_j$ that are larger in value than all elements $\pi_i$ lying to the
left.  It is never optimal to accept a candidate that is not a left-to-right
maximum because $N$ is always a left-to-right maximum in any
permutation.  The {\bf inversions} of $\pi$ consist of pairs $\pi_i > \pi_j$ where $i <
j$.

Now, let $c: \SN_N \rightarrow \N$ be some statistic on the symmetric group of
permutations of size $N$.  Then we can weight the permutation $\pi \in \SN_N$ by
$\t^{c(\pi)}$, where $\t$ is a positive real number, to obtain a discrete
probability distribution on $\SN_N$.  Distributions of this form were introduced
by Mallows where $c(\pi)$ represents some measure of distance from a fixed
permutation, typically the identity.  More recently, these distributions have
been used by researchers in combinatorics; see \cite{elizalde,bouvel,log-book}
for example.  

The {\bf weighted game of best choice} selects a permutation $\pi \in \SN_N$
with probability proportional to $\t^{c(\pi)}$ and then proceeds just as in the
classical model:  relative rankings are presented to the interviewer
sequentially and the game is won if the best candidate is accepted.  When $\t =
1$, we recover the uniform distribution so our model includes the classical game
as a specialization.  Also, we remark that for a given statistic, $c(\pi)$, the
probability of winning the weighted game of best choice under the non-uniform
distribution we defined is the same (up to rescaling by a constant involving
$\t$ and $N$) as the expected value of a random variable representing the
non-uniform payoff $\t^{c(\pi)}$ for the classical game played on a uniform
distribution.  So the weight can be interpreted as determining the underlying
probability distribution (with uniform payoff) or as determining a payoff (with
uniform probability).

In our first result, Theorem~\ref{t:main}, we show that for a large class of
statistics the optimal strategy in the weighted game of best choice has the same
form as that for the classical game:  to reject an initial set of candidates and
accept the next left-to-right maximum thereafter.  We refer to this as a {\bf
positional strategy} because the strategy accepts or rejects a candidate (that
is a left-to-right maximum) based solely on its position, even though the full
history of relative rankings containing more refined information is available.  

Next, we present an analysis of the weighted game of best choice for two
specific statistics.  The first model uses the Ewens distribution where $c(\pi)$
is the number of left-to-right maxima in $\pi$.  Setting $\t > 1$ is a concise
way of obtaining candidate ranks that tend to be improving over time, consistent
with intrinsic learning.  The second model is based on what has become known as
the Mallows distribution where $c(\pi)$ is the number of inversions in $\pi$.
Setting $\t < 1$ dampens the probability of (or imposes a cost for) experiencing
``disappointing pairs,'' where an earlier candidate ranks higher than a later
candidate.  In terms of permutation patterns, the Mallows model weights each
permutation $\pi$ by the number of $21$-instances in $\pi$, which facilitates
comparison with results in \cite{fowlkes} and \cite{jones18} for the
$321$-avoiding model.  

Once we know by Theorem~\ref{t:main} that some positional strategy is optimal
for each value of $\theta$, we can define the {\bf strategy function}
$\k_N(\theta): \mathbb{R}_{>0} \rightarrow \{0, 1, \cdots, N-1\}$ for a weighted
game of best choice to be the number of candidates that we initially reject in
the optimal strategy for value $\theta$.  In Corollary~\ref{c:ewensmain}, we
describe this function precisely for the Ewens model.  As $N$ becomes large, we
find that the optimal strategy depends on $\t$ with the optimal number of
initial rejections being $k = N/e^{1/\t}$.  Remarkably though, the probability
of success is always $1/e$, independent of $\t$, neither better nor worse than
the classical case.

The Mallows model is more subtle.  When $\t < 1$ and $N \rightarrow \infty$, the
optimal strategy is to reject all but the last $j = \max(-1/(\ln \t), 1)$
candidates and select the next left-to-right maximum thereafter.  This
``right-justified'' strategy succeeds with probability $j \t^{j-1} (1-\t)$.
When $\t > 1$, the optimal strategy is ``left-justified,'' rejecting the first
$k$ candidates for some $k$ depending on $\t$ but independent of $N$.
Consequently, this shows that the classical model (as embedded in the Mallows
model) is highly unstable!  Even an infinitesimal change away from $\t = 1$,
where the asymptotic optimal strategy rejects about $37\%$ of the candidates,
results in an optimal strategy that asymptotically rejects either $0\%$ or
$100\%$ of the candidates.  This shows that ``policy advice'' derived from the
classical model (e.g. \cite{golden}) has limited durability, a point which does
not seem to have been appreciated in popular accounts of the secretary problem.
In the future, it would be interesting to find (or better understand
obstructions to) best choice models where the asymptotic optimal strategy varies
continuously with parameterizations of the underlying distribution.

We now outline the rest of the paper.  In Section~\ref{s:setup} we review the
form of the optimal strategy for games of best choice, and show in
Section~\ref{s:tilt} that any sufficiently local statistic will generate a game
with an optimal strategy having the same form as the classical model.  In
Sections~\ref{s:ew} and \ref{s:aew}, we obtain precise and asymptotic results
for the Ewens model.  The Mallows model is treated in Section~\ref{s:m}.

%%%%%%%%%%%%%%%%%%%%%%%%%%%%%%%%%%%%%%%%%%%%%%%%%%%%%%%%%%%%%%%%%%%%%
%  Section
%%%%%%%%%%%%%%%%%%%%%%%%%%%%%%%%%%%%%%%%%%%%%%%%%%%%%%%%%%%%%%%%%%%%%
\section{Strategies for the weighted game of best choice}\label{s:setup}

In this section, we give precise notation for the various strategies that can be
employed in the weighted game of best choice, and compute their probabilities of
winning.

Fix a discrete probability distribution on the symmetric group $\SN_N$ where
$f(\pi)$ is the probability of the permutation $\pi~\in~\SN_N$.  
In this work, we are primarily interested in probability distributions obtained
from weighting by some statistic $c : \SN_N \rightarrow \N$ and a
positive real number $\theta$ via
\[ f(\pi) = \frac{ \theta^{c(\pi)} }{ \sum_{\pi \in \SN_N} \theta^{c(\pi)} }. \]
For example, we may take $c(\pi)$ to be the number of left-to-right maxima in
$\pi$, obtaining the {\bf Ewens distribution}; if $c(\pi)$ is the number of
inversions then we obtain the {\bf Mallows distribution}.
When $\theta = 1$, we recover the complete uniform distribution.

\begin{definition}\label{d:pflat}
Given a permutation $\pi$, define the {\bf $i$th prefix flattening}, denoted
$\pf{i}$, to be the unique permutation in $\SN_i$ having the same relative order
as the sequence of entries $\pi_1, \pi_2, \ldots, \pi_i$.  For brevity, we also
refer to these permutations as {\bf prefixes} of $\pi$.
\end{definition}

The set of all possible prefixes from $\bigcup_{i=1}^N \SN_i$ are partially
ordered by containment, where a permutation of smaller size occurs as the
prefix flattening of a permutation of larger size.  At the start of the game,
some $\pi$ is chosen randomly with probability $f(\pi)$.  During each move, the
$i$th prefix flattening of $\pi$ is presented to the interviewer, and they must
use only this information to determine whether to accept or reject candidate
$i$.  Therefore, we can represent any strategy (including the optimal strategy)
as a complete list of the prefixes that will immediately trigger an acceptance
by the interviewer.  We refer to such a list as a {\bf strike set}.  If, during
the game, a prefix flattening from $\pi$ does not appear in the strike set then
the current candidate is rejected and the game continues.

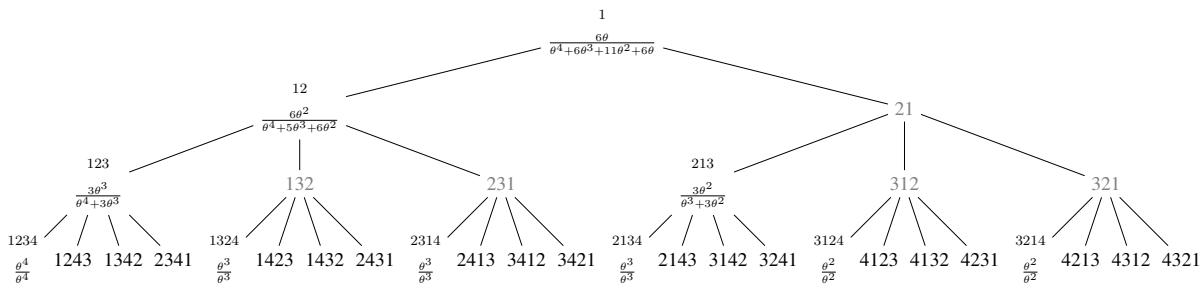
\begin{figure}[t]
\[ \scalebox{0.67}{ \begin{tikzpicture}[grow=down]
     %\tikzstyle{level 1}=[sibling distance=60mm]
     %\tikzstyle{level 2}=[sibling distance=20mm]
     %\tikzstyle{level 3}=[sibling distance=5mm]
     \tikzstyle{level 1}=[sibling distance=120mm]
     \tikzstyle{level 2}=[sibling distance=40mm]
     \tikzstyle{level 3}=[sibling distance=10mm]
     \node {$\substack{1 \\ \ \\ \frac{6\t}{\t^4 + 6\t^3 + 11\t^2 + 6\t}}$ }
     child{ node {$\substack{12 \\ \ \\ \frac{6\t^2}{\t^4 + 5\t^3 + 6\t^2}}$ } 
     child{ node {$\substack{123 \\ \ \\ \frac{3\t^3}{\t^4 + 3\t^3}}$ }
     child{ node{$\substack{1234 \\ \ \\ \frac{\t^4}{\t^4}}$}}
       child{ node{1243}}
       child{ node{1342}}
       child{ node{2341}}
   }
   child{ node[gray]  {132 }
   child{ node{$\substack{1324 \\ \ \\ \frac{\t^3}{\t^3}}$} }
       child{ node{1423}}
       child{ node{1432}}
       child{ node{2431}}
   }
   child{ node[gray]  {231 }
   child{ node{$\substack{2314 \\ \ \\ \frac{\t^3}{\t^3}}$} }
       child{ node{2413}}
       child{ node{3412}}
       child{ node{3421}}
   }
 }
 child{ node[gray]  {21 } 
 child{ node {$\substack{213 \\ \ \\ \frac{3\t^2}{\t^3 + 3\t^2}}$ }
 child{ node{$\substack{2134 \\ \ \\ \frac{\t^3}{\t^3}}$}}
       child{ node{2143}}
       child{ node{3142}}
       child{ node{3241}}
   }
   child{ node[gray]  {312 }
   child{ node{$\substack{3124 \\ \ \\ \frac{\t^2}{\t^2}}$} }
       child{ node{4123}}
       child{ node{4132}}
       child{ node{4231}}
   }
   child{ node[gray]  {321 }
   child{ node{$\substack{3214 \\ \ \\ \frac{\t^2}{\t^2}}$} }
       child{ node{4213}}
       child{ node{4312}}
       child{ node{4321}}
   }
 };
\end{tikzpicture}
}
\]
\caption{Prefix tree for $N=4$ with strike probabilities for the Ewens
distribution}\label{f:treeex}
\end{figure}

Since the set of prefixes represent all possible positions in the
game, and cover relations in the containment partial order represent ``reject''
moves by the interviewer, we can view this structure as a combinatorial game
tree that we refer to as {\bf prefix tree}.  See Figure~\ref{f:treeex} for a
small example.  The prefix tree can be solved to obtain the optimal strategy
via backwards induction as we will see below.

\begin{definition}
Given $p \in \bigcup_{i=1}^N \SN_i$, we say that $\pi \in \SN_N$ is {\bf
$p$-prefixed} if one of its prefix flattenings is $p$.  We say $\pi \in
\SN_N$ is {\bf $p$-winnable} if accepting the prefix flattening $p$ would win
the game with interview order $\pi$.  Explicitly for $p = p_1 p_2 \cdots p_k$, we have
that $\pi$ is $p$-winnable if $\pi$ is $p$-prefixed and $\pi_k = N$.  
For each prefix $p$, we define the {\bf strike probability} $\S(p)$ to be 
\[ \mathsf{P}[ \text{win the game $\pi$ under the strategy ``accept prefix $p$''} \ |\  \text{$\pi$ is $p$-prefixed} ], \]
i.e. the probability of winning the game if the strike set contained $p$ (and no other
prefixes of $p$) and we restricted the game to those interview rank orders $\pi$
having $p$ as a prefix.  
\end{definition}

Since each denominator of $f(\pi)$ is just a normalizing constant, we may cancel
it obtaining the formula
\[ \S(p) = \left( \sum\limits_{p\text{-winnable } \pi \in \SN_N} \theta^{c(\pi)}
\right)\Big/\left( \sum\limits_{p\text{-prefixed } \pi \in \SN_N}
\theta^{c(\pi)} \right). \]

\begin{definition}
We say that a prefix $p$ is {\bf eligible} if it ends in a left-to-right maximum or has size $N$.
A strike set is {\bf valid} if it 
\begin{enumerate}
\item consists of prefixes that are eligible, and 
\item has no pair of elements such that one contains the other as a prefix, and 
\item every permutation in $\SN_N$ contains some element of the strike set as a prefix.
\end{enumerate}
\end{definition}
In Figure~\ref{f:treeex} we have illustrated the prefix tree with ineligible
prefixes shown in gray and strike probabilities given below each eligible
prefix.  It follows immediately from the definitions that any strategy for the
weighted game of best choice can be represented by a valid strike set.  

Finally, we determine the optimal strategy in the form of a strike set.

\begin{definition}
Let $\S^\c(p)$ be the {\bf open} probability of a win if we play optimally,
using any strike set consisting of prefixes that contain (but are not equal to)
$p$.  Similarly, let $\bar\S({p})$ be the {\bf closed} probability of a win if
we play optimally, using any strike set consisting of prefixes that contain (and
may include) $p$.  That is,
\[ \S^\c(p) = \mathsf{P}[ \text{win the game $\pi$ under the best strategy available after rejecting $p$ } \ |\  \text{$\pi$ is $p$-prefixed} ], \]
\[ \bar\S(p) = \mathsf{P}\left[ \parbox{4in}{win the game $\pi$ under the best strategy available after rejecting the largest proper prefix of $p$ } \ \Big|\  \text{$\pi$ is $p$-prefixed} \right], \]
We define each of these to be conditional probabilities,
restricting to those interview orders $\pi$ having $p$ as a prefix, with {\bf
standard denominator} $\sum_{p\text{-prefixed } \pi \in \SN_N} \theta^{c(\pi)}$.  
\end{definition}

Then, it follows directly from the definitions that
\[ \bar\S({p}) = \max(\S(p), \S^\c(p)). \]
This formula can be used to recursively determine $\bar\S({1})$, the globally
optimal probability of a win.  To also keep track of the optimal strategy, let
us say that a prefix $p$ is {\bf positive} if $\S(p) \geq \S^\c(p)$ and {\bf
negative} otherwise.  %It is locally optimal to accept a positive prefix.

\begin{example}
Suppose $N = 4$ as illustrated in Figure~\ref{f:treeex} and $\theta = 1$.  Consider the prefix $p = 123$, so the first three candidates have increasing ranks.  The best we can do after rejecting them is to accept the last candidate.  This succeeds in winning the game with probability $\S^\c(123) = 1/4$.  If, by contrast, we were also allowed to accept the third candidate, then we should do so as our probability of winning becomes $\bar\S(123) = 3/4$.
\end{example}

\begin{theorem}
With the setup given above, a globally optimal strike set for a weighted game of
best choice consists of the subset $A$ of positive prefixes that are minimal when
partially ordered by prefix-containment.  The probability of winning is 
$\bigoplus_{p \in A} \S(p)$ where we define $\frac{a}{b} \oplus \frac{c}{d} =
\frac{a+c}{b+d}$ and use the standard denominator for all strike probabilities.
\end{theorem}
\begin{proof}
We first explain how to determine all of the strike probabilities in the prefix
tree using the formulas.  The prefixes of size $N$ are positive, which serves
as a base case for induction on the size of a prefix.  Given the probabilities
for prefixes of sizes greater than $i$, the $\S^\c(p)$ probabilities for each
prefix of size $i$ can be obtained as $\bigoplus_{p\text{-prefixed } q \text{
of size }i+1}\bar\S({q})$, and then the $\bar\S({p})$ probabilities can be
determined from the $\max$ formula.  (The $\oplus$ operation represents the
probability of winning on a disjoint union of sample spaces.)  This process is
essentially a combinatorial version of ``backwards induction.''

The positive prefixes $p$ are locally optimal (for the collection of $p$-prefixed
$\pi \in \SN_N$) by definition.  If a prefix $p$ has no proper prefix flattenings
that are positive, then $p$ must be part of the globally optimal strike set as well.
\end{proof}

\begin{example}
Consider the tree shown in Figure~\ref{f:treeex} for $N = 4$ at $\t = 1$.  The
strike probabilities $\S(p)$ are illustrated in the figure.  Since $\S(123) =
3/4 \geq 1/4 = \S^\c(123)$, we have that $123$ is a positive prefix.
Similarly, $\S(12) = 6/12 \geq 5/12 = \S^\c(12)$ so $12$ is also a positive
prefix.  On the other hand, $\S(1) = 6/24 < 11/24 = \S^\c(1)$ so $1$ is a
negative prefix.  The optimal strike set consists of $12$ (contributing
$6/12$), $213$ (contributing $3/4$), $3124$ (contributing $1/1$), and $3214$
(contributing $1/1$), together with the six prefixes of size $4$ that aren't
already related to one of these, each contributing a strike probability of
$0/1$.  The optimal probability is the $\oplus$-sum of these contributions,
namely $11/24$.
\end{example}

%%%%%%%%%%%%%%%%%%%%%%%%%%%%%%%%%%%%%%%%%%%%%%%%%%%%%%%%%%%%%%%%%%%%%
%  Section
%%%%%%%%%%%%%%%%%%%%%%%%%%%%%%%%%%%%%%%%%%%%%%%%%%%%%%%%%%%%%%%%%%%%%
\section{Prefix equivariance and positional strategies}\label{s:tilt}

While any game of best choice has an optimal strike strategy, the classical game
(where $f(\pi) = 1/N!$ uniformly) is optimized by a positional strategy in which
the player rejects the first $k$ candidates and accepts the next left-to-right
maximum thereafter.  In this section, we give a concrete explanation for this
and generalize it to a class of weighted games.  The key idea is to use a
fundamental bijection in order to transport structure around the prefix tree.  

\begin{definition}
Let $\T^\c(p)$ be the {\bf open}
subforest of prefixes containing (but not equal to) $p$, and let $\bar\T({p})$ be the
{\bf closed} subtree of prefixes containing $p$, including $p$ itself.  
\end{definition}

\begin{definition}
Suppose that $p = [12 \cdots k]$ and let $\sigma_q$ be the permutation action that
rearranges the prefix $p$ to give some other prefix $q$ of size $k$.  We can
extend this to an action on $\bar\T({p})$, denoted $\sigma_q \cdot \pi$, by
similarly permuting the first $k$ entries and fixing the last $m-k$ entries of
$\pi$, where $m$ is the size of $\pi \in \bar\T({p})$.  Then $\sigma_q$ is a
bijection from $\bar\T({p})$ to $\bar\T({q})$.
\end{definition}

\begin{definition}
Suppose that the statistic $c$ satisfies $c(\pi) - c(\sigma_q \cdot \pi) = c(12
\cdots k) - c(q)$ for all prefixes $q$ and all $\pi \in \bar\T({12 \cdots
k})$.  Here, $k$ is the size of $q$.  Then, we say that $c$ is a {\bf prefix
equivariant} statistic.
\end{definition}

This condition essentially says that the change in the statistic $c(\pi)$ that
results from permuting the first $k$ entries of $\pi$ is the same as the change
that would result if we restricted $c$ to $k$ entries.  That is, the action of
permuting a prefix does not create or destroy any structure being counted by
$c$, beyond the entries of the prefix.  Hence, statistics that count
sufficiently local phenomena in permutations will be prefix equivariant.

\begin{example}
It is straightforward to check that $c(\pi) = \#$~left-to-right maxima in $\pi$
and $c(\pi) = \#$~inversions in $\pi$ are each prefix equivariant statistics.
Explicitly, if $\pi = \pi_1 \pi_2 \cdots \pi_k | \pi_{k+1} \pi_{k+2} \cdots
\pi_m$ has the form of an increasing block of size $k$ followed by an arbitrary
block, then we may observe that rearranging the first block may change the
number of left-to-right maxima within that block, but it cannot change the
left-to-right maximal status of any entry in the second block.  Similarly,
rearranging the first block may change the number of inversions within that
block, but it cannot add or remove an inversion pair where the smaller entry
lies in the second block.  On the other hand, $c(\pi) = \#321$-instances in
$\pi$ (i.e. triples of entries that are decreasing) is not prefix equivariant
because, for example, $c(2468|1357) = 0$ yet $c(4268|1357) = 1$ even though
$c(2468) = 0 = c(4268)$.
\end{example}

The following results are illustrated in Figure~\ref{f:treeex}.

\begin{theorem}\label{t:pres}
Suppose $c(\pi)$ is a prefix equivariant statistic.
For all prefixes $q$, the strike probabilities are
preserved under the restricted bijection $\sigma_q: \T^\c(12 \cdots k)
\rightarrow \T^\c(q)$, where $k$ is the size of $q$.  If $q$ is eligible then
these probabilities are also preserved under $\sigma_q: \bar\T({12 \cdots
k}) \rightarrow \bar\T({q})$.

Consequently, for $p \in \T^\c(12 \cdots k)$ (and additionally for $p = [12 \cdots k]$ if $q$ is
eligible), we have that
\begin{itemize}
    \item the $\S^\c(p)$ probabilities are preserved by $\sigma_q$, and
    \item the $\bar\S({p})$ probabilities are preserved by $\sigma_q$, and
    \item if $p$ and $q$ are eligible, we have $p$ is positive if and only if $\sigma_q \cdot p$ is positive.
\end{itemize}
\end{theorem}
\begin{proof}
Fix $q$ to be any prefix of size $k$, and let $p \in \bar\T({12 \cdots k})$ with size $m$.  Then, 
$\S(p) = \frac{ \sum_{p\text{-winnable } \pi \in \SN_N} \theta^{c(\pi)} }{
\sum_{p\text{-prefixed } \pi \in \SN_N} \theta^{c(\pi)} }$.
Since the weights satisfy $c(\pi) - c(\sigma_q \cdot \pi) = c(12 \cdots k) -
c(q)$ for all $\pi \in \bar\T({12 \cdots k})$, we have 
\[ \S(\sigma_q \cdot p) = \frac{ \sum_{(\sigma_q \cdot p)\text{-winnable } \pi
\in \SN_N} \theta^{c(\pi)} }{ \sum_{(\sigma_q \cdot p)\text{-prefixed } \pi
\in \SN_N} \theta^{c(\pi)} } = \frac{ \sum_{p\text{-winnable } \pi \in \SN_N}
\theta^{c(\sigma_q \cdot \pi)} }{ \sum_{p\text{-prefixed } \pi \in \SN_N} \theta^{c(\sigma_q \cdot \pi)} } \]
which is $\frac{ \theta^{c(q) - c(12 \cdots k)} }{ \theta^{c(q) - c(12 \cdots
k)}} S(p)$ unless $q$ happens to be ineligible and $p = [12 \cdots k]$ in which case
$\sigma_q \cdot p = q$ and so $\S(q) = 0$.

The probability $\S^\c(p)$ is an $\oplus$-sum of strike probabilities, say
$\S^\c(p) = \S(r_1) \oplus \S(r_2) \oplus \cdots \oplus \S(r_n)$,
for some prefixes $r_i \in \bigcup_{j = m+1}^N \SN_j$.
Then, 
\[ \S^\c(\sigma_q \cdot p) = \S(\sigma_q \cdot r_1) \oplus \S(\sigma_q \cdot
r_2) \oplus \cdots \oplus \S(\sigma_q \cdot r_n) \]
\[     = \frac{ \theta^{c(q) - c(12 \cdots k)} }{ \theta^{c(q) - c(12 \cdots
k)}} \S(r_1) \oplus \frac{ \theta^{c(q) - c(12 \cdots k)} }{ \theta^{c(q) - c(12
\cdots k)}} \S(r_2) \oplus \cdots \oplus \frac{ \theta^{c(q) - c(12 \cdots k)}
}{ \theta^{c(q) - c(12 \cdots k)}} \S(r_n) \]
which is $\frac{ \theta^{c(q) - c(12 \cdots k)} }{ \theta^{c(q) - c(12 \cdots
k)}} \S^\c(p)$.

The other consequences now follow from applying the recursive formulas.
\end{proof}

Thus, it suffices to restrict our attention to subtrees lying under increasing
prefixes.  To avoid clutter in the remainder of the results, we abuse notation
to let $\S(p), \S^\c(p)$, and $\bar\S({p})$ each refer to their numerators over
the standard denominator $\sum_{p\text{-prefixed } \pi \in \SN_N}
\theta^{c(\pi)}$.  To ensure that this is valid, we avow that all of our
equalities will occur between quantities for which their implied denominators
agree.  

\begin{theorem}\label{t:pc}
For the increasing prefixes $p = [12 \cdots (k-1)]$ and $q = [12 \cdots k]$, we have
\[ \S^\c(p) = \bar\S({q}) + \S^\c(q) \sum_{\substack{\text{nontrivial
permutations} \\r \text{ of } q \text{ having } p \text{ as a prefix }}} \theta^{c(r)-c(q)}, \text{ and } \]
\[ \S(p) = \S(q) \sum_{\substack{\text{nontrivial permutations} \\r \text{ of }
q \text{ having } p \text{ as a prefix }}} \theta^{c(r)-c(q)}. \]
\end{theorem}
\begin{proof}
There are $k$ children of $p$ in the prefix tree; they are distinguished by
their value in the last position.  The prefix $q$ itself is an eligible child
so $\bar\S({q})$ is the optimal probability for the subtree rooted at $q$.  The subtrees
under each of the other $k-1$ children of $p$ are isomorphic to $\T^\c(q)$ via
the bijection $\sigma_r$ where $r$ is a nontrivial permutation of $q$ having
$p$ as a prefix.  For each $\pi \in \SN_N$ that is won in $\S^\c(q)$, we have
$\theta^{c(\sigma_r \cdot \pi)} = \theta^{c(\pi)} \theta^{c(r)-c(q)}$ by prefix
equivariance, and the first result follows.

The second result is similar.  First, observe that none of the $q$-prefixed
$\pi \in \SN_N$ are $p$-winnable.  Each of the $p$-winnable permutations arises
by applying one of the $\sigma_r$ to a $q$-winnable permutation $\pi$.  This
has the effect of placing the value $N$ into position $k-1$, as desired.
Prefix equivariance produces a factor of $\theta^{c(r)-c(q)}$ for each choice
of $\sigma_r$.
\end{proof}

\begin{corollary}\label{c:t1}
For any increasing prefixes $p = [12 \cdots (k-1)]$ and $q = [12 \cdots k]$, we
have that if $q$ is negative then $p$ is negative.
\end{corollary}
\begin{proof}
Suppose $q$ is negative so $\S^\c(q) > \S(q)$.  Then, by
Theorem~\ref{t:pc} we have
\[ \S^\c(p) = \bar\S({q}) + \S^\c(q) \sum_{\substack{\text{nontrivial permutations} \\r \text{ of } q \text{ having } p \text{ as a prefix }}} \theta^{c(r)-c(q)} \]
\[ > \S(q) \left( 1 + \sum_{\substack{\text{nontrivial permutations} \\r \text{ of } q \text{ having } p \text{ as a prefix }}} \theta^{c(r)-c(q)} \right) =  \S(q) + \S(p)  \geq \S(p), \]
so $p$ is negative as well.
\end{proof}

\begin{theorem}\label{t:main}
For a weighted game of best choice defined using a prefix equivariant
statistic, the optimal strategy is positional.
\end{theorem}
\begin{proof}
By Corollary~\ref{c:t1}, there exists some $k$ such that all of the increasing
prefixes with size less or equal to $k$ are negative, and all of the increasing
prefixes with size greater than $k$ are positive.  Applying the $\sigma_q$
isomorphisms, the same $k$ also serves to separate positive and negative
eligible prefixes in the rest of the tree by Theorem~\ref{t:pres}.  Hence, the
optimal strike strategy coincides with the positional strategy that rejects the
first $k$ candidates and accepts the next left-to-right maximum thereafter.
\end{proof}

\begin{remark}\label{r:use}
Knowing that the optimal strategy in a game of best choice is positional
simplifies the analysis considerably.  Rather than recursively traversing the
entire prefix tree to find the optimal strike set, which could conceivably be
any antichain, we know that the optimum must be one of the $N$ strategies
obtained by rejecting $k = 0, 1, 2, \ldots$, or $N-1$ candidates and accepting
the next maximum.  Then, one can write an equation to determine $k$.
\end{remark}

\begin{remark}\label{r:hist}
It is interesting to compare the proof above at $\t = 1$ with arguments from the
literature on the classical secretary problem.  While numerous papers either
assume or claim that the optimal strategy is positional, few of them actually
provide or reference a complete proof.  As far as we can tell,
there are essentially two elementary arguments for the result.  One appears as
part of a “reminiscence” in \cite{kadison} that develops a careful proof using
convexity, along different lines than ours, but has not been widely cited
(although Ron Graham is acknowledged in the postscript).  The other is a
conditioning argument, apparently originating in \cite{palermo} (unpublished)
and very loosely summarized by \cite{gilbert--mosteller}.  Once symmetry (as in
Theorem~\ref{t:pres}) is established, this argument compares the conditional
probabilities for winning if we accept the current candidate (``stay'') versus
if we consider eligible candidates later in the game (``go'').  The argument
claims that the ``stay'' probabilities are increasing, which can be verified
directly by counting. The argument also claims that the ``go'' probabilities are
decreasing, but counting is recursive in this case (corresponding to our
Theorem~\ref{t:pc}) and is complicated by the fact that ``stay'' and ``go'' are
not complementary probabilities.
\end{remark}

%%%%%%%%%%%%%%%%%%%%%%%%%%%%%%%%%%%%%%%%%%%%%%%%%%%%%%%%%%%%%%%%%%%%%
%  Section
%%%%%%%%%%%%%%%%%%%%%%%%%%%%%%%%%%%%%%%%%%%%%%%%%%%%%%%%%%%%%%%%%%%%%
\section{Precise results for Ewens distribution}\label{s:ew}

When we weight by $c(\pi) = \#$ left-to-right maxima in $\pi$, we obtain the
Ewens distribution.  In this section, we work out the optimal best choice
strategy for all $N$.

\begin{definition}
Let $\stir{N}$ be the polynomial in $\t$ defined by 
\[ \stir{N} =  \t(\t+1)(\t+2) \cdots (\t+(N-1)). \]
\end{definition}

The following result justifies our ``$\t$-analogue'' notation.

\begin{lemma}
We have
\[ \stir{N} = \sum_{\pi \in \SN_N} \theta^{\# \text{left-to-right maxima in } \pi}. \]
Hence, the coefficients of $\t$ in $\stir{N}$ are Stirling numbers (of the type used to count permutations by number of cycles).
\end{lemma}
\begin{proof}
This is straightforward to prove using induction since we may extend each
permutation $\pi$ of $N-1$ by placing one of the values $1, 2, \ldots, N-1$ in
the last position and arranging the complementary values according to $\pi$
(this does not create a new left-to-right maximum so contributes $(N-1)
\stir{N-1}$ to $\stir{N}$) or by simply appending the value $N$ to the last
position of $\pi$ (which does create a new left-to-right maximum, so
contributes $\t \stir{N-1}$ to $\stir{N}$).

The equivalence between the number of cycles and number of left-to-right maxima
(attributed to R\'enyi) is accomplished by writing the cycle notation for a
permutation using the maximum element in a cycle as the starting point and then
arranging the cycles with increasing maximum elements.  
\end{proof}

We are primarily interested in 
\[ W(N,k) = \sum_{k\text{-winnable } \pi \in \SN_N} \t^{\#\text{left-to-right maxima in } \pi}. \]
Here, we say that $\pi$ is {\bf $k$-winnable} if it would be won by the
positional strategy that rejects the first $k$ candidates and accepts the next
left-to-right maximum thereafter.  Some examples of these polynomials are given
in Figure~\ref{f:W}.  Our next result provides a recursive description for them.

\begin{theorem}\label{t:W}
We have
\[ W(N,k) = (N-1) W(N-1,k) + \frac{(N-2)!}{(k-1)!} \t \stir{k} \]
with initial conditions $W(1, 0) = \t$ and $W(N,N) = 0$.
\end{theorem}
\begin{proof}
We have two cases for the $k$-winnable permutations $\pi \in \SN_{N}$.
\begin{itemize}
    \item If the last position contains one of the values $1, 2, \ldots, N-1$,
        then it is not a left-to-right maximum and 
        we may view the complementary values as some $k$-winnable
        $\widetilde{\pi} \in \SN_{N-1}$.  Hence, these
        contribute $(N-1) W(N-1,k)$ to $W(N,k)$.
    \item If the last position contains $N$ then it is a left-to-right maximum
        and the value $N-1$ must lie in one of the first $k$ positions in order
        for $\pi$ to be $k$-winnable.  We can choose the rest of the values to
        place among the first $k$ positions in ${ {N-2} \choose {k-1}}$ ways and
        then permute them, keeping track of the number of left-to-right maxima
        with $\stir{k}$.  For each of these, we may also then permute the rest
        of the entries in positions $k+1, k+2, \ldots, N-1$ in $(N-k-1)!$ ways.
        All together, these contribute
        \[ \t^1 { {N-2} \choose {k-1}} \stir{k} (N-k-1)! = \frac{(N-2)!}{(k-1)!}
        \t\stir{k} \]
        to $W(N,k)$.
\end{itemize}
The initial conditions are immediate.
\end{proof}

\begin{figure}[t]
    \scalebox{0.8}{
$\begin{array}{llllllllllllll}
N & k=0 & k=1 & k=2 & k=3 & k=4 & k=5 \\
\hline \\
1 & \t & \\
2 & \t & \t^2 & \\
3 & 2\t & 3\t^2 & \t^3 + \t^2 & \\
4 & 6\t & 11\t^2 & 5\t^3 + 5\t^2 & \t^4 + 3\t^3 + 2\t^2 & \\
5 & 24\t & 50\t^2 & 26\t^3 + 26\t^2 & 7\t^4 + 21\t^3 + 14\t^2 & \t^5 + 6\t^4 + 11\t^3 + 6\t^2 & \\
6 & 120\t & 274\t^2 & 154\t^3 + 154\t^2 & 47\t^4 + 141\t^3 + 94\t^2 & 9\t^5 + 54\t^4 + 99\t^3 + 54\t^2 & \t^6 + 10\t^5 + 35\t^4 + 50\t^3 + 24\t^2 & \\
\end{array}$ }
\caption{Some $W(N,k)$ polynomials}\label{f:W}
\end{figure}

Now, let $\DW(N,k) = W(N,k+1) - W(N,k)$.  The zeros of these polynomials will determine
the intervals of $\t$ that produce games for which a given positional strategy is
optimal.  We begin by translating the recurrence.

\begin{corollary}\label{c:Wrec}
We have
\[ \DW(N,k) = (N-1) \DW(N-1,k) + \t^2 \frac{(N-2)!}{k!} \stir{k} \]
with initial conditions $\DW(N,N-2) = \t(\t-(N-1))\stir{N-2}$.
\end{corollary}
\begin{proof}
By Theorem~\ref{t:W}, we have
\[ \DW(N,k) = W(N,k+1) - W(N,k) = (N-1) \left(W(N-1,k+1) - W(N-1,k)\right)  \]
\[    + \t (N-2)! \left( \frac{ \stir{k+1} }{ k! } - \frac{\stir{k}}{(k-1)!} \right)  \]
\[ = (N-1) \DW(N-1,k) + \t \frac{(N-2)!}{k!} \stir{k} \left( (\t + k) - k \right)  \]
yielding the result.

The initial conditions follow by subtracting $W(N,N-2) = (2N-3) \t \stir{N-2}$
from $W(N,N-1) = \t \stir{N-1}$.
\end{proof}

It turns out that we can solve this recurrence.

\begin{theorem}\label{t:dwmain}
We have 
\[ \DW(N,k) = c_1(N,k) \left( \left(\sum_{i=k+1}^{N-1} \frac{1}{i}\right)\t - 1 \right) \t \stir{k} \]
for $c_1(N,k) = \prod_{j=k+1}^{N-1} j$ which is constant in $\t$.  Hence, $\DW$ has only real roots.
Moreover, the only positive root of $\DW(N,k)$ occurs at
\[ \t = 1 \Big/ \left(\sum_{i=k+1}^{N-1} \frac{1}{i}\right). \]
\end{theorem}
\begin{proof}
We fix $k$ and argue by induction on $N$.  When $N = k-2$, we apply the initial
condition in Corollary~\ref{c:Wrec} with $c_1(N,N-2) = (N-1)$.

Now suppose the result holds for $N-1$.  Then, by Corollary~\ref{c:Wrec}
\[ \DW(N,k) = (N-1) \left(  c_1 \left( \left(\sum_{i=k+1}^{N-2} \frac{1}{i}\right)\t - 1 \right) \t \stir{k} \right) + \t^2 \frac{(N-2)!}{k!} \stir{k} \]
\[ = \left(  (N-1) c_1 \left( \left(\sum_{i=k+1}^{N-2} \frac{1}{i}\right)\t - 1 \right) + \t \frac{(N-2)!}{k!} \right) \t \stir{k} \]
\[ = (N-1) c_1 \left( \left( \frac{(N-2)(N-3) \cdots (k+1)}{(N-1) c_1} + \frac{\sum\limits_{k+1 \leq i \leq N-2}\ \ \prod\limits_{\substack{k+1 \leq j \leq N-2 \\ j \neq i}} j }{(N-2)(N-3) \cdots (k+1)} \right) \t - 1 \right) \t \stir{k}. \]
Now, if we let $c_1 = (N-2)(N-3) \cdots (k+1)$, we may rewrite the linear term as
\[ \left( \frac{(N-2)(N-3) \cdots (k+1)}{(N-1)(N-2)(N-3) \cdots (k+1)} + \frac{(N-1)\sum\limits_{k+1 \leq i \leq N-2}\ \ \prod\limits_{\substack{k+1 \leq j \leq N-2 \\ j \neq i}} j }{(N-1)(N-2)(N-3) \cdots (k+1)} \right) \t - 1  = \left( \sum_{i=k+1}^{N-1} \frac{1}{i}\right)\t - 1 \]
obtaining $(N-1) c_1 \left( \left( \sum_{i=k+1}^{N-1} \frac{1}{i}\right)\t - 1
\right) \t \stir{k}$ as desired.
\end{proof}

\begin{corollary}\label{c:ewensmain}
We have
\[ \k_N(\t) = \begin{cases} 
        0 & \text{ if } 0 < \t \leq \left(\sum_{i=1}^{N-1} \frac{1}{i}\right)^{-1} \\
        k & \text{ if }  \left(\sum_{i=k}^{N-1} \frac{1}{i}\right)^{-1} < \t \leq \left(\sum_{i=k+1}^{N-1} \frac{1}{i}\right)^{-1} \\
        N-1 & \text{ if } \t > N-1.
\end{cases} \]
\end{corollary}
\begin{proof}
Since, for fixed $N$, the positive roots from Theorem~\ref{t:dwmain} are unique
and increasing in $k$, we find that the strategy function $\k_N(\t)$ is increasing as
well.
\end{proof}

This completely determines the optimal strategy precisely for all $N$.  Some of
the cutoff values for $\t$ are illustrated in Figure~\ref{f:cr}.  We have
highlighted the optimal range including $\t = 1$ corresponding to the classical
uniform case.

\begin{figure}[th]
\scalebox{0.8}{
$\begin{array}{llllllllllllll}
N & k=0 & k=1 & k=2 & k=3 & k=4 & k=5 & k=6 & k=7 & k=8 & k=9 \\ %& k=10 \\
\hline \\
2 & {\bf 1*} & \\
3 & 2/3 & {\bf 2*} & \\
4 & 6/11 & {\bf 6/5*} & 3 & \\
5 & 12/25 & 12/13 & {\bf 12/7*} & 4 & \\
6 & 60/137 & 60/77 & {\bf 60/47*} & 20/9 & 5 & \\
7 & 20/49 & 20/29 & {\bf 20/19*} & 60/37 & 30/11 & 6 & \\
8 & 140/363 & 140/223 & 140/153 & {\bf 420/319*} & 210/107 & 42/13 & 7 & \\
9 & 280/761 & 280/481 & 280/341 & {\bf 840/743*} & 840/533 & 168/73 & 56/15 & 8 & \\
10 & 2520/7129 & 2520/4609 & 2520/3349 & {\bf 2520/2509*} & 2520/1879 & 504/275 & 504/191 & 72/17 & 9 & \\
11 & 2520/7381 & 2520/4861 & 2520/3601 & 2520/2761 & {\bf 2520/2131*} & 2520/1627 & 2520/1207 & 360/121 & 90/19 & 10 & \\
\end{array}$ }
\caption{Some critical roots $(\sum_{i=k+1}^{N-1}(1/i))^{-1}$}\label{f:cr}
\end{figure}

%%%%%%%%%%%%%%%%%%%%%%%%%%%%%%%%%%%%%%%%%%%%%%%%%%%%%%%%%%%%%%%%%%%%%
%  Section
%%%%%%%%%%%%%%%%%%%%%%%%%%%%%%%%%%%%%%%%%%%%%%%%%%%%%%%%%%%%%%%%%%%%%
\bigskip
\section{Asymptotic results for Ewens distribution}\label{s:aew}

To facilitate a comparison with the classical case, we can also solve
the Ewens model asymptotically.

\subsection{Optimal strategy}
For fixed $\t$ and large $N$, the optimal $k$ is given by solving
\[ \t^{-1} = \sum_{i=k}^{N-1} \frac{1}{i} = \sum_{i=k}^{N-1} \frac{1}{i/(N-1)} \frac{1}{N-1} \]
for $k$.  The latter is a Riemann sum approximation for the integral $\int_x^1
\frac{1}{t} \ dt$ where $t = \frac{i}{N-1}$ and $x = \frac{k}{N-1}$.  

Therefore, as $N \rightarrow \infty$ we obtain
\[ \t^{-1} = \int_x^1 \frac{1}{t} \ dt = -\ln x \]
which we can solve for $x = 1/e^{1/\t}$.  Thus, the optimal number of initial
rejections is approximately $k = N/e^{1/\t}$ for $N$ sufficiently large.  A plot is
shown in Figure~\ref{f:4}.

\begin{figure}[h]
\[ \includegraphics[scale=0.3]{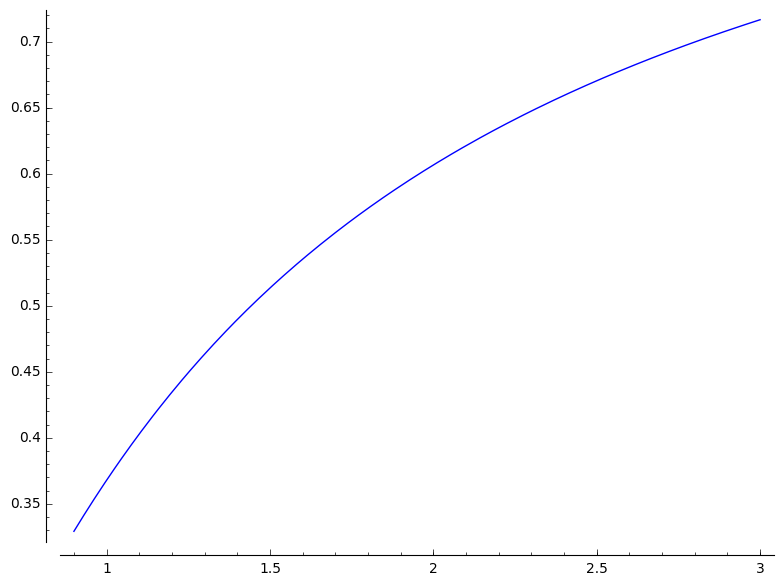} \]
\caption{Proportion of initial rejections in the Ewens distribution for $\t > 1$}\label{f:4}
\end{figure}

\bigskip
\subsection{Optimal probability of success}
Reviewing the previous section, we find that we can solve $W(N,k)$ explicitly.
Each $W(N,k)$ polynomial is just a constant in $\t$ (that depends on $N$ and
$k$) times the polynomial $\t \stir{k}$.

\begin{theorem}\label{t:eds}
For all $N$ and $k \geq 1$, we have
\[ W(N,k) = \t \stir{k} \frac{(N-1)!}{(k-1)!} \sum_{i=k}^{N-1} \frac{1}{i}. \]
When $k = 0$, we have $W(N,0) = (N-1)! \t$.
\end{theorem}
\begin{proof}
This is straightforward to prove by induction from Theorem~\ref{t:W}.
\end{proof}

Since the optimal value of $k$ satisfies $\t = \left(\sum_{i=k}^{N-1}
\frac{1}{i}\right)^{-1}$ we can cancel it and the optimal probability $W$
simplifies to
\[ \frac{W(N,k)}{\stir{N}} = \frac{\stir{k}}{(k-1)!} \frac{(N-1)!}{\stir{N}}. \]

When $\t$ is an integer, this is just a ratio of binomial coefficients, but for
arbitrary positive real $\t$ we use gamma functions (see e.g. \cite{NIST}).  By iterating the
recurrence $\G(z+1)=z\G(z)$ we obtain
$\stir{N} = \frac{\G(\t+N)}{\G(\t)}$.
Hence,
\[ \frac{\stir{k}}{(k-1)!} \frac{(N-1)!}{\stir{N}}
= \frac{\G(\t+k)}{\G(\t)(k-1)!} \frac{\G(\t)(N-1)!}{\G(\t+N)}
= \frac{\G(\t+k)}{\G(k)} \frac{\G(N)}{\G(\t+N)}. \]
Using $\lim_{x \rightarrow \infty} \frac{\G(x+\t)}{\G(x) x^\t} = 1$, we get
\[ \lim_{N \rightarrow \infty} \frac{\G(\t+k)}{\G(k)} \frac{\G(N)}{\G(\t+N)}
= \lim_{N \rightarrow \infty} \frac{k^\t}{N^\t} 
= \lim_{N \rightarrow \infty} \left(\frac{k}{N}\right)^{\t}  
= \left(\frac{1}{e^{1/\t}}\right)^\t = 1/e. \]
Remarkably, this probability of success is independent of $\t$.

%%%%%%%%%%%%%%%%%%%%%%%%%%%%%%%%%%%%%%%%%%%%%%%%%%%%%%%%%%%%%%%%%%%%%
%  Section
%%%%%%%%%%%%%%%%%%%%%%%%%%%%%%%%%%%%%%%%%%%%%%%%%%%%%%%%%%%%%%%%%%%%%
\section{Results for Mallows distribution}\label{s:m}

We now turn to the Mallows distribution defined by $c(\pi) = \#$inversions in
$\pi$.  We begin by working out the standard ``$\t$-analogue'' for this
statistic.

\begin{definition}
Let $\ta{N}$ be the polynomial in $\t$ defined by $1+\t+\t^2+\cdots+\t^{N-1}$.
Let $\fa{N}$ be the polynomial in $\t$ defined by 
\[ \fa{N} =  \ta{N} \ta{N-1} \cdots \ta{1}. \]
\end{definition}

\begin{lemma}
We have
\[ \fa{N} = \sum_{\pi \in \SN_N} \theta^{\# \text{inversions in } \pi}. \]
\end{lemma}
\begin{proof}
This is straightforward to prove using induction since we may extend each
permutation $\pi$ of $N-1$ by placing one of the values $i = 1, 2, \ldots, N$
in the last position and arranging the complementary values according to $\pi$.
This creates $N-i$ new inversions, so contributes $\ta{N} \fa{N-1}$ to $\fa{N}$.
\end{proof}

Let us redefine
\[ W(N,k) = \sum_{k\text{-winnable } \pi \in \SN_N} \t^{\#\text{inversions in } \pi} \]
for the Mallows distribution.  Our next result provides a recursive description
for these polynomials.

\begin{theorem}\label{t:mW}
We have
\[ W(N,k) = \t \ta{N-1} W(N-1,k) + \t^{N-k-1} \ta{k} \fa{N-2} \]
with initial conditions $W(1, 0) = 1$ and $W(N,N) = 0$.
\end{theorem}
\begin{proof}
We have two cases for the $k$-winnable permutations $\pi \in \SN_{N}$.
\begin{itemize}
    \item If the last position contains one of the values $i = 1, 2, \ldots,
        N-1$, then it contributes $N-i$ to the inversion count and we may view
        the complementary values as some $k$-winnable $\widetilde{\pi} \in
        \SN_{N-1}$.  Hence, these contribute $\t \ta{N-1} W(N-1,k)$ to $W(N,k)$.
    \item If the last position contains $N$ then 
        the value $N-1$ must lie in one of the first $k$ positions in order
        for $\pi$ to be $k$-winnable.  
        These choices for the position of value $N-1$ contribute $\t^{N-k-1} \ta{k}$.
        For each of these, we choose a permutation of size $N-2$ to fill in the
        remaining positions, keeping track of the inversions with $\fa{N-2}$.
\end{itemize}
The initial conditions are immediate.
\end{proof}

We can solve this recurrence.

\begin{corollary}\label{c:mm}
We have
\[ W(N,k) = \t^{N-k-1} \fa{N-1} \sum_{i=k}^{N-1} \frac{ \ta{k} }{ \ta{i} } \]
for $k > 0$ and $W(N,0) = \t^{N-1} \fa{N-1}$.
\end{corollary}
\begin{proof}
This is straightforward to prove by induction on $N$ from Theorem~\ref{t:mW}.
\end{proof}

For this distribution, the precise transition probabilities for each $N$ seem
to be inaccessible, being roots of polynomials (with complex solutions) that
use many repeated root extractions as opposed to the rational numbers we
obtained in the Ewens case.  However, we obtain some interesting asymptotic
results.  Figure~\ref{f:m1} shows a plot of the optimal success probability for
various values of $\t < 1$ based on the following theorem.

\begin{figure}[t]
\[ \includegraphics[scale=0.3]{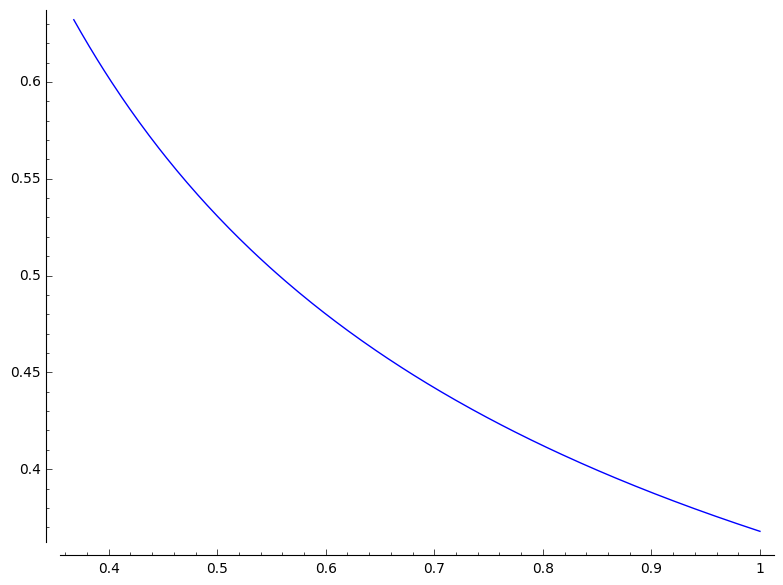} \]
\caption{Success probability in the Mallows distribution for $\t < 1$}\label{f:m1}
\end{figure}

\begin{theorem}
If $\t < 1$ then the optimal strategy as $N$ becomes large is to reject the
first $N-j$ candidates where $j = \max(-1/(\ln \t), 1)$, and select the next
left-to-right maximum thereafter.  This strategy succeeds with probability $j
\t^{j-1} (1-\t)$.  If $1/e < \t < 1$, this probability of success simplifies to
$e^{-1} \left(\frac{\t - 1}{\t \ln \t}\right)$.
\end{theorem}
\begin{proof}
Rewriting the probability $W(N,k)/\fa{N}$ from Corollary~\ref{c:mm} gives
\[ \t^{N-k-1} \frac{ \ta{k} }{ \ta{N} } \sum_{i=k}^{N-1} \frac{1}{\ta{i}} =
(1-\t) \t^{N-k-1} \frac{1-\t^k}{1-\t^N} \sum_{i=k}^{N-1} \frac{ 1 }{ 1-\t^i }  \]
\[ = \frac{\t^N-\t^{N-1}-\t^{N-k}+\t^{N-k-1}}{1 - \t^N} \sum_{i=k}^{N-1} \frac{ 1 }{
1-\t^i }. \]

For fixed $\t \in (0,1)$, as $N$ becomes large, the fraction tends to
$\t^{N-k-1}(1 - \t)$ and the terms in the series become close to $1$ so the
probability reduces to $\t^{N-k-1}(1-\t)(N-k)$.  This sequence converges to a
positive value if and only if the $N-k$ sequence converges to a finite value.

So we let $j = N-k$ obtaining $\lim_{N \rightarrow \infty} W(N,k)/\fa{N}
= j \t^{j-1} (1-\t)$.  Differentiating with respect to $j$ and setting equal to
zero, we solve to obtain the optimal $j = -1/(\ln \t)$.  For $\t < 1/e$, we
have $j < 1$ but we cannot reject more than $N-1$ candidates so max appears
in the expression.
\end{proof}

The case where $\t > 1$ is less interesting from our intrinsic learning
perspective but we sketch the behavior of these models for mathematical
completeness.  Taken together, the results also prove that the asymptotically
optimal strategy does not vary continuously with the parameter $\t$. 

\begin{figure}[t]
\[ \includegraphics[scale=0.25]{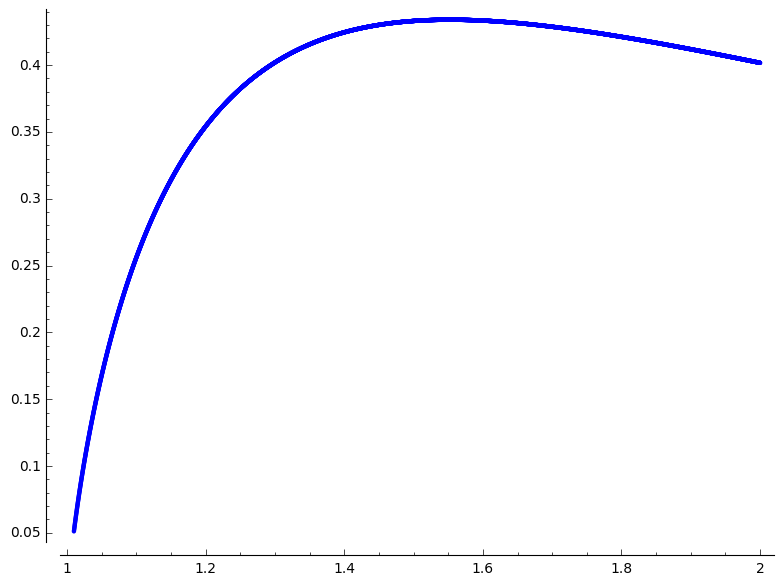} \includegraphics[scale=0.25]{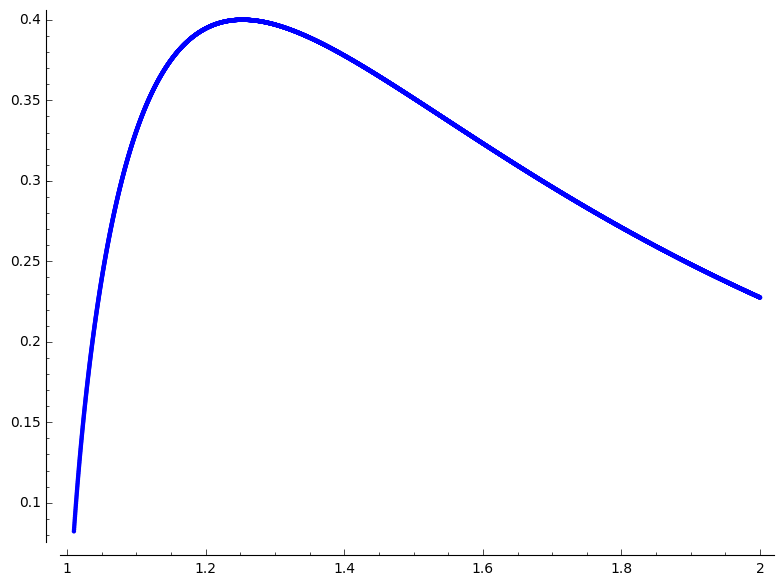} \includegraphics[scale=0.25]{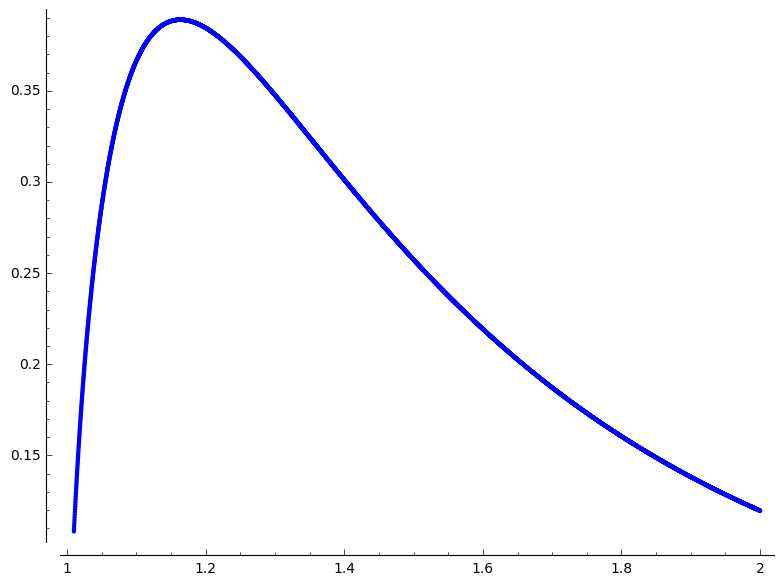} \]
\caption{The success probability for $\t > 1$ in the cases $k = 1$, $k = 2$, and $k = 3$.
The maximum occurs at $(1.55, 0.433939)$ for $k = 1$, $(1.25, 0.400125)$ for $k = 2$, and $(1.16, 0.389029)$ for $k = 3$.}\label{f:m2}
\end{figure}

\begin{corollary}\label{c:mm2}
Fix $\t > 1$ and an integer $k \geq 1$.  Then,
\[ \lim_{N \rightarrow \infty} \frac{W(N,k)}{\fa{N}} = (\t-1) \frac{\t^k - 1}{\t^{k+1}} \sum_{i=k}^{\infty} \frac{1}{\t^i - 1}. \]
In particular, we approach a single asymptotic distribution as $N \ra \infty$
so the optimal strategy will be to reject some number of initial candidates
(not depending on $N$) and select the next left-to-right maximum thereafter.
\end{corollary}
\begin{proof}
Once again, consider 
\[ \frac{W(N,k)}{\fa{N}} = \frac{\t^N-\t^{N-1}-\t^{N-k}+\t^{N-k-1}}{\t^N - 1} \sum_{i=k}^{N-1} \frac{ 1 }{ \t^i - 1 }. \]
For fixed $\t > 1$, as $N$ becomes large, the fraction tends to
$1-\t^{-1}-\t^{-k}+\t^{-(k+1)} = (1-\t^{-1})(1 - \t^{-k})$ and the series
converges (e.g. by the integral test).  

Thus, every $k$-positional strategy has some nonzero probability in the limit
as $N \rightarrow \infty$.  As $k$ becomes large, the series tends to zero,
being approximately $\sum_{i=k}^\infty (1/\t)^i = \t^{-k}/(1-\t^{-1})$.
Therefore, the optimal asymptotic probability will occur for some fixed $k$
(depending on $\t$).
\end{proof}

Figure~\ref{f:m2} shows some plots (obtained numerically) of the optimal
success probabilities for various values of $\t$ and $k$.  For any particular
value of $\t > 1$, one can explicitly compute the probabilities in
Corollary~\ref{c:mm2} and find the optimal strategy $k$.  For example, it
appears that $k = 1$ is optimal for all $\t > 1.285$ (approximately).  Although
the series always converges, finding a closed form for its limiting value
involves the ``$q$-digamma function'' which prevents us from obtaining a simple
description of the optimal strategy in general.  It also appears that the
maximal probability of success for these models occurs when $\t \approx 1.55$
(for which $k = 1$ is the optimal strategy).  It would be interesting to
determine this maximum more precisely.

%%%%%%%%%%%%%%%%%%%%%%%%%%%%%%%%%%%%%%%%%%%%%%%%%%%%%%%%%%%%%%%%%%%%%
%  Acknowledgements
%%%%%%%%%%%%%%%%%%%%%%%%%%%%%%%%%%%%%%%%%%%%%%%%%%%%%%%%%%%%%%%%%%%%%
\section*{Acknowledgements}

We thank Laura Taalman, John Webb, and Becky Wild for helpful discussions related to this
work.

%%%%%%%%%%%%%%%%%%%%%%%%%%%%%%%%%%%%%%%%%%%%%%%%%%%%%%%%%%%%%%%%%%%%%
% ==========   Bibliography
%%%%%%%%%%%%%%%%%%%%%%%%%%%%%%%%%%%%%%%%%%%%%%%%%%%%%%%%%%%%%%%%%%%%%

%%\bibliographystyle{alpha}
%\bibliographystyle{amsalpha}
%\bibliography{bestchoice}

\providecommand{\bysame}{\leavevmode\hbox to3em{\hrulefill}\thinspace}
\providecommand{\MR}{\relax\ifhmode\unskip\space\fi MR }
\providecommand{\MRhref}[2]{%
  \href{http://www.ams.org/mathscinet-getitem?mr=#1}{#2}
}
\providecommand{\href}[2]{#2}

\end{document}